\documentclass[12pt, a4paper]{amsart}
\usepackage{amsmath,amssymb,amsfonts, float}
\usepackage[all]{xy}
\usepackage{caption}
\usepackage{enumerate}
\usepackage{mathpazo}
\usepackage{a4wide}
\usepackage{diagbox}
\usepackage{subcaption}

\usepackage{bm}% bold math
\usepackage{graphicx}
\usepackage[breaklinks=true]{hyperref}

\usepackage{xcolor}
\usepackage{multicol}
\usepackage{tabularx}

\usepackage{hyperref}
\usepackage[capitalize]{cleveref}

\usepackage[normalem]{ulem}

\newtheorem{thm}{Theorem}[section] 
\newtheorem*{thm*}{Theorem} 
\newtheorem{prop}[thm]{Proposition}
\newtheorem{lem}[thm]{Lemma}
\newtheorem{cor}[thm]{Corollary}

\theoremstyle{definition}
\newtheorem{definition}[thm]{Definition}
\newtheorem{expl}[thm]{Example}

\newtheorem{question}[thm]{Question}
\newtheorem{rem}[thm]{Remark}

\DeclareMathOperator{\C}{\mathbb{C}}
\DeclareMathOperator{\Z}{\mathbb{Z}}

\DeclareMathOperator{\F}{\mathbb{F}}

\DeclareMathOperator{\Hom}{{\rm Hom}}

\newcommand{\Ann}{{\rm Ann}}

    \DeclareFontFamily{U}{wncy}{}
    \DeclareFontShape{U}{wncy}{m}{n}{<->wncyr10}{}
    \DeclareSymbolFont{mcy}{U}{wncy}{m}{n}
    \DeclareMathSymbol{\Sha}{\mathord}{mcy}{"58}

\numberwithin{equation}{section}

\renewcommand{\i}{\mathrm{i}}

\newcommand{\soc}{{{\rm soc}}}
\newcommand{\Isoc}{{{\rm Isoc}}}

\DeclareSymbolFont{bbold}{U}{bbold}{m}{n}
\DeclareSymbolFontAlphabet{\mathbbold}{bbold}

\usepackage{hyperref}

\begin{document}
\title{Perfect state transfer on \\ gcd-graphs over a finite Frobenius ring}
 \author{Tung T. Nguyen, Nguy$\tilde{\text{\^{e}}}$n Duy T\^{a}n }

 \address{Department of Mathematics and Computer Science, Lake Forest College, Lake Forest, Illinois, USA}
 \email{tnguyen@lakeforest.edu}
 
  \address{
Faculty Mathematics and Informatics, Hanoi University of Science and Technology, 1 Dai Co Viet Road, Hanoi, Vietnam } 
\email{tan.nguyenduy@hust.edu.vn}

\thanks{TTN is partially supported by an AMS-Simons Travel Grant.  NDT is partially supported by the Vietnam National
Foundation for Science and Technology Development (NAFOSTED) under grant number 101.04-2023.21}


\maketitle
\begin{abstract}
The existence of perfect state transfer (PST) on quantum spin networks is a fundamental problem in mathematics and physics. Various works in the literature have explored PST in graphs with arithmetic origins, such as gcd-graphs over $\mathbb{Z}$ and cubelike graphs. In this article, building on our recent work on gcd-graphs over an arbitrary finite Frobenius ring, we investigate the existence of PST on these graphs. Our approach is algebraic in nature, enabling us to unify various existing results in the literature.
\end{abstract}
\section{\textbf{Introduction}}
Let $G$ be an undirected graph with adjacency matrix $A_G$. Let $F(t)$ be the \textit{continuous-time} quantum walk associated with $G$; namely $F(t) = \exp(\mathrm{i}A_Gt).$ There is perfect state transfer (PST) in graph $G$ if there are distinct vertices $a$ and $b$ and a positive real number $t$ such that $|F(t)_{ab}| = 1$. In other words, \(G\) is said to exhibit PST if it admits the transfer of a quantum state between two vertices without loss of information under the continuous-time quantum walk. The concept of PST was first introduced in \cite{landahn1, landahn}, where the authors investigated the problem of arranging \(N\) interacting qubits in a network that allows the perfect transfer of an arbitrary quantum state over the greatest possible distance. Since these pioneering works, the subject has generated extensive literature, particularly within the quantum computing and algebraic combinatorics communities (see \cite{godsil2012state} for some fundamental results about PST on graphs, and \cite{arnadottir2022state} for a survey on the state of the art of the problem). We also refer to \cite{BP1, BP2, BP3, pst-gcd-new-2, pst-gcd-new} for some foundational studies on this line of research. We note that some researchers define PST on $G$ using another matrix describing the \textit{discrete-time} quantum walk on $G$ (for example, see \cite{bhakta2025state}). In this paper, we follow Godsil's definition as stated above. It is interesting to note that $F(t)$ also appears quite unexpectedly in the complex-valued Kuramoto model, a standard model in non-linear dynamics that describes the synchronization of phase oscillators (\cite{budzinski2021simple, muller2021algebraic}).  By \cite[Theorem 5.2]{godsil2012state}, it is known that if a regular graph $G$ has PST, then $G$ must be integral; i.e, all of its eigenvalues must be integers. In particular, the class of Cayley graphs that have PST is a subset of the class of integral Cayley graphs.

Let us briefly recall a special class of integral graphs that have a long history in algebraic graph theory; namely the gcd-graphs introduced by Klotz and Sander in \cite{klotz2007some}. Let $n$ be a positive integer and $D$ a subset of proper divisors of $n$. The gcd-graph $G_n(D)$ is the graph with the following data.
\begin{enumerate}
\item The vertices of $G_n(D)$ are elements of the finite ring $\Z/n$.
\item Two vertices $a,b$ are adjacent if $\gcd(a-b, n) \in D.$ 
\end{enumerate}
Using the theory of Ramanujan sums, Klotz and Sander explicitly describe the spectrum of $G_n(D)$ (see \cite[Section 4]{klotz2007some}). Various works in the literature have utilized this explicit spectral description to study perfect state transfer (PST) on gcd-graphs (see, for example, \cite{bavsic2009perfect, godsil2012state, saxena2007parameters}).

In \cite{minavc2024gcd}, inspired by the analogy between number fields and function fields, we introduce the notion of gcd-graphs over $\F_q[x]$ as an analog of the gcd-graphs over $\Z$. Using the theory of Ramanujan and Gauss sums developed by Lamprecht in \cite{lamprecht1953allgemeine}, we are also able to describe the spectra of these gcd-graphs. Furthermore, we provide the necessary and sufficient conditions for a gcd-graph over $\F_q[x]$ to be integral (see \cite{nguyen2024integral}). We soon realize that the definition of gcd-graphs can in fact be defined over an arbitrary ring (we work out the details in \cite{nguyen2025gcd}). When the underlying ring $R$ is a finite Frobenius ring, their natural duality allows us to relate spectra of these generalized gcd-graphs to certain Ramanujan sums which can then be calculated explicitly (see \cite[Theorem 4.16]{nguyen2025gcd}). A direct consequence of this spectral analysis is that the gcd-graph $\Gamma(R,S)$ is integral, making it a promising candidate for graphs that have PST. We note that the integrality of $\Gamma(R,S)$ also follows from \cite[Theorem 1.1]{nguyen2024integral}, although that theorem does not provide a concrete spectral description. In this article, we develop an algebraic and number-theoretic approach to the problem of PST on gcd-graphs. As we will explain later in the article, this approach is more streamlined and, therefore, capable of unifying various existing results in the literature.

This is our first paper on this line of research. Here, we focus on establishing the fundamental foundations and applying them to study local Frobenius rings. The more general case, in the spirit of \cite{BP1, BP2, BP3}, will be discussed in our future work.

\subsection{Outline}
In \cref{sec:spectral}, we describe the spectral theory for $R$-circulant matrices, where $R$ is a finite Frobenius ring. We then apply this theory to establish a concrete criterion for the existence of perfect state transfer (PST) on a gcd-graph defined over $R$ (see \cref{prop:condition}). Utilizing this criterion, we develop in \cref{sec:necessary_conditions} some necessary conditions for the existence of PST on a gcd-graph. In particular, we demonstrate that PST can exist only between specific elements in $R$ (see \cref{prop:upper_bound}). In \cref{sec:local_rings}, we focus on the case where $R$ is a local Frobenius ring. Here, we identify a sufficient condition for a gcd-graph to possess PST (see \cref{thm:local}). Furthermore, in certain instances, we show that this condition is also necessary, and hence we completely classify gcd-graphs with PST in those cases. Finally, in \cref{sec:unitary}, we classify all unitary Cayley graphs that exhibit PST (see \cref{thm:unitary}).

\subsection{Code}
The code that we wrote to generate data and perform experiments with PST can be found in \cite{nguyen_pst}.  We remark that we have also verified all statements in this work with various concrete and computable examples.

\section{\textbf{Spectral analysis of Cayley graphs over a finite Frobenius ring}}  \label{sec:spectral}
Let $R$ be a finite Frobenius ring and $n$ a positive integer such that $nR = 0$. We will fix $n$ throughout this section, since we only need its exact value when describing the character group of $R$. Most of the time, $n$ will only appear implicitly.  In previous work \cite{nguyen2024integral, nguyen2025gcd}, we used the term \textit{finite symmetric $\mathbb{Z}/n$-algebra} to refer to $R$. However, to maintain consistency with the literature, we will adopt the term Frobenius ring instead. The equivalence between these two seemingly different definitions follows from the work of Lamprecht in \cite{lamprecht1953allgemeine} (see \cite{honold2001characterization} for historical insights and various ring and module theoretic characterizations of Frobenius rings).

By definition, being a Frobenius ring means that $R$ is a $\mathbb{Z}/n$-algebra equipped with a nondegenerate linear functional $\psi\colon R \to \mathbb{Z}/n$. This non-degenerate functional is crucial in our spectral analysis. In fact, its existence implies that the character group of $(R, +)$ is parametrized by $R$ itself, allowing us to treat $R$ similarly to $\mathbb{Z}/n$ (see \cite[Proposition 2.4]{nguyen2024integral}). Consequently, as explained below, several, although not all, techniques from \cite{bavsic2009perfect, godsil2012state} can be naturally adapted to this more general setting.

Let $G$ be a Cayley graph defined over $R$, namely $G= \Gamma(R,S)$ for some symmetric subset $S$ of $R.$ By definition, the adjacency matrix $A_G$ is an $R$-circulant matrix (see \cite{kanemitsu2013matrices, chebolu2022joins}). The entries of $A_G$ are uniquely determined by the first row $\vec{c}=[c_s]_{s \in R}$ where 
\[
c_r=
\begin{cases}
1 & \text{if } r \in S, \\
0 & \text{otherwise. } 
\end{cases}
\]
Since $R$ is a Frobenius ring and $nR=0$, every character $\widehat{\psi}$ of $R$ is of the form $\widehat{\psi}(s)= \zeta_n^{\psi(rs)}$ for some unique $r \in R$, where $\zeta_n=e^{2\pi\i/n}$ (see \cite{lamprecht1953allgemeine, nguyen2024integral}).  For such $r \in R$, let 
\[ \vec{v}_r = \frac{1}{\sqrt{|R|}}[\zeta_n^{\psi(rs)}]_{s \in R}^{T} \in \mathbb{C}^{|R|} .\] 
\begin{rem}
By \cite[Proposition 10]{nguyen2023equilibria}, for each $r \in R$,  $\vec{\theta}_r = \left(\frac{2 \pi \psi(rs)}{n} \right)_{s \in R}$ is an equilibrium point of the Kuramoto model on $G = \Gamma(R,S).$ It would be interesting to study the stability of these points when $S$ varies (see \cite{townsend2020dense} for an example of this investigation). 
\end{rem}

By the circulant diagonalization theorem (see \cite{kanemitsu2013matrices}), the set $\{\vec{v}_r : r \in R\}$ forms a normalized eigenbasis for all $A_G$. Furthermore,  the eigenvalues of $A_G$ are precisely
\[ \lambda_r = \sum_{s \in R} c_s \zeta_{n}^{\psi(rs)} = \sum_{s \in S} \zeta_n^{\psi(rs)}.\]
We note that, by definition, $\lambda_r$ depends on $S$. However, since $S$ is often clear from the context, we will omit it in the definition of $\lambda_r.$ 
Let $V = [v_r]_{r \in R} \in \C^{|R| \times |R|}$ be the matrix formed by this eigenbasis and $V^{*}$ be the conjugate transpose of $V$. Then we can write
\[ A_G = V D V^*=\sum_{r \in R} \lambda_r \vec{v}_r\vec{v}_r^{*},\]
here $D =\text{diag}([\lambda_r]_{r \in R})$ is the diagonal matrix formed by the eigenvalues $\lambda_r.$ We then have
\[ F(t)= \sum_{r \in R} e^{\mathrm{i} \lambda_r t} \vec{v}_r \vec{v}_r^{*}. \]
Hence
\[ F(t)_{s_1, s_2} = \frac{1}{|R|} \sum_{r \in R} e^{\mathrm{i} \lambda_r t} \zeta_n^{\psi((s_1-s_2)r)} = \frac{1}{|R|} \sum_{r \in R} e^{2 \pi \mathrm{i} \left(\lambda_r \frac{t}{2 \pi}+\frac{ \psi((s_1-s_2)r)}{n}\right)}. \]
By the triangle inequality, $|F(t)_{s_1, s_2}|=1$ if and only if $\lambda_r \frac{t}{2 \pi} + \frac{\psi(s_1-s_2)r}{n}$ are constant modulo $1.$ Furthermore, by symmetry, there exists perfect state transfer between $s_1$ and $s_2$ if and only if there exists perfect state transfer between $0$ and $s_2-s_1.$ In summary, our calculation shows the following criterion, which is a direct generalization of \cite[Theorem 4]{bavsic2009perfect}.
\begin{thm} \label{prop:condition}
    There exists perfect state transfer from $0$ to $s$ at time $t$ if and only if for all $r_1, r_2 \in R$
    \[ (\lambda_{r_1}-\lambda_{r_2}) \frac{t}{2 \pi} + \frac{\psi(s(r_1-r_2))}{n} \equiv 0 \pmod{1}.\]
\end{thm}

\section{\textbf{Necessary conditions for the existence of PST}} \label{sec:necessary_conditions}
In this section, we apply \cref{prop:condition} to impose some necessary conditions for gcd-graphs defined over $R$ to have PST. First, we recall their definition, which was introduced in \cite{nguyen2025gcd}.
\begin{definition}
Let $R$ be a finite Frobenius ring. Let $D= \{\mathcal{I}_1, \mathcal{I}_2, \ldots, \mathcal{I}_k \}$ be a set of distinct principal ideals in $R.$ The gcd-graph $G_{R}(D)$ is equipped with the following data.
\begin{enumerate}
    \item The vertex set of $G_{R}(D)$ is $R.$
    \item Two vertices $a,b$ are adjacent if there exists $\mathcal{I} \in D$ such that $R(a-b)= \mathcal{I}$ where $R(a-b)$ is the principal ideal generated by $a-b.$
\end{enumerate}
\end{definition}
We remark that a Cayley graph $\Gamma(R,S)$ defined over $R$ with a generating set $S$ is a gcd-graph if and only if $S$ is stable under the action of $R^{\times}$ (see \cite[Proposition 2.5]{nguyen2025gcd}.)  

Let $\Gamma(R,S) = G_{R}(D)$ be a gcd-graph. Let $\Delta$ be the abelian group generated by $r_1-r_2$ where $r_1$ and $r_2$ are elements of $R$ such that $\lambda_{r_1} = \lambda_{r_2}$. 
The following statement is a direct corollary of \cref{prop:condition}.

\begin{cor} \label{cor:orthogonal}
Let $\Gamma(R,S) = G_{R}(D)$ be a gcd-graph.   Suppose that there exists  perfect state transfer from $0$ to $s$ at time $t$. Then $\psi(sd)=0$ for all $d \in \Delta.$
\end{cor}

\begin{proof}
\cref{prop:condition} implies that if there exists perfect state transfer from $0$ to $s$ at time $t$ then $\psi(s(r_1-r_2))=0.$ Since $\psi$ is linear, we conclude that $\psi(sd)=0$ for all $d \in \Delta.$
\end{proof}
Let us now discuss a lower bound for the size of $\Delta.$ By \cite[Corollary 4.17]{nguyen2025gcd}, $\lambda_{r_1}=\lambda_{r_2}$ if $r_1$ and $r_2$ are associates. In particular, $\Delta$ contains the subgroup $\Delta'$ generated by $u_1-u_2$ where $u_1, u_2$ are units in $R.$ Let us now describe the structure of $\Delta'.$ To do so, we recall that by the structure theorem for artinian rings, $R \cong \prod_{i=1}^d R_i$ where each $R_i$ is a local ring. We can further write $R=R_1 \times R_2$ where $R_1$ (respectively, $R_2$) consists of local factors whose residue fields are $\F_2$ (respectively, $\neq \F_2$). Let $J(R_1)$ be the Jacobson radical of $R_1$. Because $R_1$ is a finite product of local rings, $J(R_1)$ is the kernel of the map $R_1 \to R_1/J(R_1) \cong \F_2^r$ where $r$ is the number of local factors whose residue fields are $\F_2.$ By \cite[Corollary 3.5]{nguyen2025gcd}, every element in $J(R_1) \times R_2$ is a difference of two units. This shows that $J(R_1) \times R_2 \subset \Delta'$. In fact, this is an equality. 
\begin{prop} \label{prop:Delta}
$\Delta' = J(R_1) \times R_2$ and hence $J(R_1) \times R_2 \subset \Delta$. Consequently, if there exists PST between $0$ and $s$ then $s \in \Ann_{R}(J(R_1) \times R_2).$
\end{prop}
\begin{proof}
    We already know that $J(R_1) \times R_2 \subset \Delta'.$ Let us also prove that $\Delta' \subset J(R_1) \times R_2$. It is enough to show that if $u_1, u_2 \in R_{1}^{\times}$ then $u_1 - u_2 \in J(R_1).$ For this, it is sufficient to assume that $R_1$ is a local ring with maximal ideal $\mathfrak{m}$ and $|R/\mathfrak{m}|=2.$ In this case, $u_1 \equiv u_2 \equiv 1 \pmod{\mathfrak{m}}$ and hence $u_1- u_2 \in \mathfrak{m}=J(R_1).$

    The last statement follows from \cref{cor:orthogonal} that $\psi(sd)=0$ for all $d \in J(R_1) \times R_2.$ This implies that $s (J(R_1) \times R_2)$ belongs to the kernel of $\psi.$ Since $\psi$ is non-degenerate, we must have $s (J(R_1) \times R_2) = 0$. In other words, $s \in \Ann_{R}(J(R_1) \times R_2).$
\end{proof}

Let us now focus on the structure of a local Frobenius ring. By \cite{honold2001characterization, lamprecht1953allgemeine}, a local ring $(S,\mathfrak{m})$ is Frobenius if and only if there exists an element $e \in \mathfrak{m}$ such that $Se$ is a minimal ideal in $S$; namely $Se$ is contained in every non-zero ideal in $S$. Said it differently, $Se$ is precisely the socle module of $S$ (see \cref{def:soc} for the definition of the socle of a module). It turns out that if the residue field of $S$ is $\F_2$, $e$ must be unique. 

\begin{lem} \label{lem:local}
    Let $(S, \mathfrak{m})$ be a local finite Frobenius ring such that the residue field of $S$ is $\F_2.$ Then, there exists a unique $e \in S$ such that $\Ann_{R}(e)=\mathfrak{m}$. Furthermore, if $\psi\colon S \to \Z/n$ is any non-degenerate linear functional then $\dfrac{\psi(e)}{n} \equiv \dfrac{1}{2} \pmod{1}.$ 
\end{lem}

\begin{proof}
    Let $e'$ be another element such that $Se=Se'.$ Then $e'=ue$ for some $u \in S^{\times}.$ Since the residue field of $S$ is $\F_2$, $u-1 \in \mathfrak{m}.$ Furthermore, since $\Ann_{S}(e)=\mathfrak{m}$, we conclude that $(u-1)e=0$. This implies that $e'=e.$

    For the second statement, we know that $\psi(e) \neq 0$ since $\psi$ is non-degenerate. Additionally, $0=\psi(2e)=2 \psi(e).$ This shows that $\psi(e) \equiv \frac{n}{2} \pmod{n}.$ Equivalently, $\frac{\psi(e)}{n} \equiv \frac{1}{2} \pmod{1}.$ 
\end{proof}

Combining \cref{prop:Delta} and \cref{lem:local} we have the following. 

\begin{thm} \label{prop:upper_bound}
Let $R$ be a finite Frobenius ring. Suppose that $R$ has the following Artin-Wedderburn decomposition: $R = (\prod_{i=1}^d S_i) \times R_2$. Here, $(S_i, \mathfrak{m_i})$ represents all local factors of $R$ whose residue fields are $\mathbb{F}_2$. For each $1 \leq i \leq d$, let $e_i$ be the unique minimal element of $S_i$ as defined in \cref{lem:local}. If there exists PST between $0$ and some $s \in R$, then $s$ must be of the form $(a_1, a_2, \ldots, a_d, 0)$, where each $a_i$ is $0$ or $e_i$. In particular, if $R$ is a local ring, then $s=e$, where $e$ is the minimal element of $R$. 
\end{thm}

\begin{proof}
    Suppose that there exists PST between $0$ and $s=(a_1, a_2, \ldots, a_d, a) \in R.$ Then by \cref{prop:Delta}, $a_i \in \Ann_{S_i}(m_i) = S_ie_i$ for each $1 \leq i \leq d.$ By \cref{lem:local}, this implies that $a_i \in \{0, e_i \}$. Additionally, $a \in \Ann_{R_2}(R_2)=0.$
\end{proof}

We discuss some corollaries of \cref{prop:upper_bound}.
\begin{cor} \label{cor:number_of_pst}
Let $t$ be a fixed time. Then, there is at most one value of $s \in R$ such that there is perfect state transfer from $0$ to $s$ at time $t$. 
\end{cor}
\begin{proof}
Let $s_1, s_2 \in R$ such that PST exists both between $0$ and $s_1$ and between $0$ and $s_2$ at time $t.$ Then \cref{prop:condition} implies that $\psi((r_1-r_2)s_1) =\psi((r_1-r_2)s_2)$ for all $r_1, r_2 \in R$ since $\lambda_{r_1}, \lambda_{r_2}$ are independent of $r_1, r_2.$ Suppose that $s_1 \neq s_2.$ Without loss of generality, suppose that the first components of $s_1, s_2$ are different, namely $s_1 =(e_1, a_2, \ldots, a_d, 0)$ and $s_2 = (0, a_2', \ldots, a_d', 0)$ as explained in \cref{prop:upper_bound}. Let $r_2=(1, 0, \ldots, 0)$ and $r_1 = (0, 0, \ldots, 0)$. We then have $
\psi((r_2-r_1)s_1) = \frac{n}{2}$ by \cref{lem:local} while $\psi((r_2-r_1)s_2) =0$. This is a contradiction. 

\end{proof}

\begin{rem}
    For each $s \neq 0$ of the form $(a_1, a_2, \ldots, a_d,0) \in  R$ where $a_i \in \{0, e_i \}$, there exists a gcd-graph $G_{R}(D)$ such that PST exists between $0$ and $s$ at time $t= \frac{\pi}{2}.$ In fact, we can take $D=\{Rs\}$. In this case, $G_{R}(D)$ consists of a disjoint union of several copies of $K_2$, with one of these being the induced graph on the vertices $0$ and $s$. Consequently, PST exists between $0$ and $s$ at time $\frac{\pi}{2}$.
\end{rem}

Here is another corollary of \cref{prop:upper_bound}.

\begin{cor}
\label{cor:no PST}
If $d=0$ then there is no perfect state transfer on $G_{R}(D).$ In particular, if $R$ is an $\F_q$-algebra with $q \neq 2$, then there is no perfect state transfer on any gcd-graphs defined over $R.$
\end{cor}

\section{\textbf{PST on gcd-graphs over a local Frobenius ring}} \label{sec:local_rings}

Let $(R, \mathfrak{m})$ be a finite local Frobenius ring where $\mathfrak{m}$ is the maximal ideal of $R$.  In this section, we study PST on gcd-graphs defined over $R$ that have PST. We remark that since $R$ is Frobenius, it has the following elegant symmetry: for each ideal $I$ in $R$, $\Ann_{R}(\Ann_{R}(I))=I$ (see \cite{honold2001characterization}). As a consequence, an ideal is uniquely determined by its annihilator ideal. 

As explained in \cref{lem:local}, since $R$ is a Frobenius ring, there exists a unique $e \in R$ such that $Re$ is the unique minimal ideal in $R$. By \cref{cor:no PST}, a gcd-graph over $R$ has PST only if the residue field $R/\mathfrak{m}$ is $\F_2.$ For this reason, we will assume this property throughout this section. 

As explained in \cref{cor:number_of_pst}, PST can only exist between $0$ and $e.$ Furthermore, when $R$ is a local ring, \cref{prop:condition} can be simplified as follows. 

\begin{prop} \label{prop:simplified_condition}
    There exists PST between $0$ and $e$ at time $t$ if and only if for all $r \in \mathfrak{m}$ the following conditions hold. 
    \begin{enumerate}
        \item $(\lambda_r - \lambda_0) \frac{t}{2 \pi} \in \Z$. 
        \item $(\lambda_r - \lambda_1) \frac{t}{2 \pi} + \frac{1}{2}  \in \Z.$
    \end{enumerate}
\end{prop}
\begin{proof}
 Suppose that there exists PST between $0$ and $e$ at time $t$, then by \cref{prop:condition}, for all $r,u\in R$, we have 
 \[ (\lambda_{r}-\lambda_{u}) \frac{t}{2 \pi} + \frac{\psi(e(r-u))}{n} \equiv 0 \pmod{1}.\]
In particular, if $r\in \mathfrak{m}$ then $e(r-0)=er=0$ and $e(r-1)=-e$ since $\Ann_{R}(e)=\mathfrak{m}$, and thus
\[(\lambda_{r}-\lambda_{0}) \frac{t}{2 \pi} + \frac{\psi(e(r-0))}{n}=(\lambda_{r}-\lambda_{0}) \frac{t}{2 \pi}\equiv 0 \pmod{1},\]
and 
\[
(\lambda_{r}-\lambda_{1}) \frac{t}{2 \pi} + \frac{\psi(e(r-1))}{n}=(\lambda_{r}-\lambda_{1}) \frac{t}{2 \pi}+\dfrac{1}{2} \equiv 0 \pmod{1}.
\]
Hence, the conditions (1) and (2) hold.

Conversely, suppose that the conditions (1) and (2) hold. We need to show that for all $r,u\in R$, 
 \[ (\lambda_{r}-\lambda_{u}) \frac{t}{2 \pi} + \frac{\psi(e(r-u))}{n} \equiv 0 \pmod{1}.\tag{*}\]
If $r$ and $u$ are not in $\mathfrak{m}$ then $r$ and $u$ are in $R^\times$ and hence $\lambda_r=\lambda_u=\lambda_1$. Also, since $R/\mathfrak{m}=\F_2$, $r-u\in \mathfrak{m}$ and $e(r-u)=0$. Thus the condition (*) is clearly true.
If $r$ and $u$ are in  $\mathfrak{m}$ then $e(r-u)=0$ and
\[
(\lambda_{r}-\lambda_{u}) \frac{t}{2 \pi} + \frac{\psi(e(r-u))}{n}=(\lambda_{r}-\lambda_{0}) \frac{t}{2 \pi} -(\lambda_{u}-\lambda_{0}) \frac{t}{2 \pi}  \equiv 0\pmod{1}.
\]
Hence the condition (*) is also true in this case.
The remaining case to consider is that exactly one of $r$ or $u$ is in $\mathfrak{m}$. We may assume that $r\in \mathfrak{m}$ and $u\not\in\mathfrak{m}$. In this case $u\in R^\times$ and hence $\lambda_u=\lambda_1$. We can also write $u=1+r'$, for some $r'\in \mathfrak{m}$. Thus $e(r-u)=e(r-r')-e=-e$. Hence 
\[
(\lambda_{r}-\lambda_{u}) \frac{t}{2 \pi} + \frac{\psi(e(r-u))}{n} =(\lambda_{r}-\lambda_{1}) \frac{t}{2 \pi} - \frac{\psi(e)}{n}=(\lambda_{r}-\lambda_{1}) \frac{t}{2 \pi}-\dfrac{1}{2}\equiv 0\pmod{1}.
\]
 The condition (*) is also true in this case.
 %   This follows from the fact that $\lambda_{u} = \lambda_1$ if $u$ is a unit and $\Ann_{R}(e)=\mathfrak{m}.$
\end{proof}

We will start our investigation with two concrete examples that will serve as a guide for what follows (see \cref{fig:combined} for their graphical representations).

\begin{figure}[ht]
    \centering
    \begin{subfigure}[t]{0.45\textwidth}
        \centering
        \includegraphics[width=\linewidth]{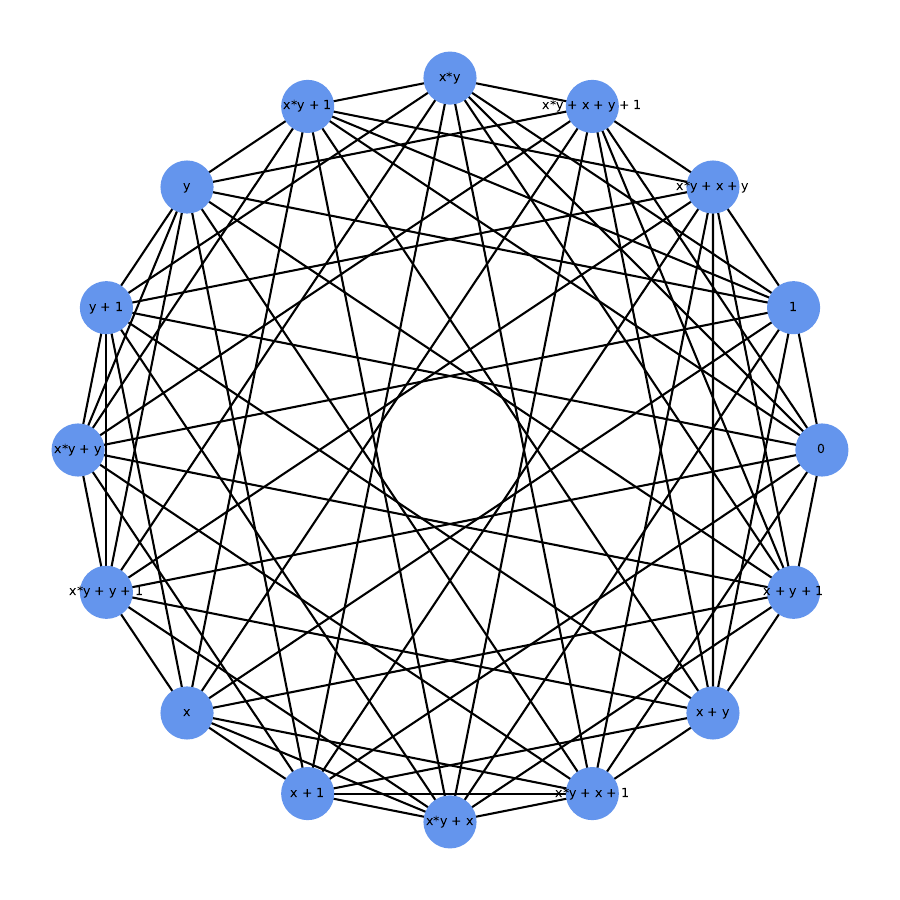} % Replace with your image file
        \caption{The gcd-graph $G_{R}(D)$ with $D=\{R, Rxy \}$}
        \label{fig:figure1}
    \end{subfigure}
    \hfill
    \begin{subfigure}[t]{0.45\textwidth}
        \centering
        \includegraphics[width=\linewidth]{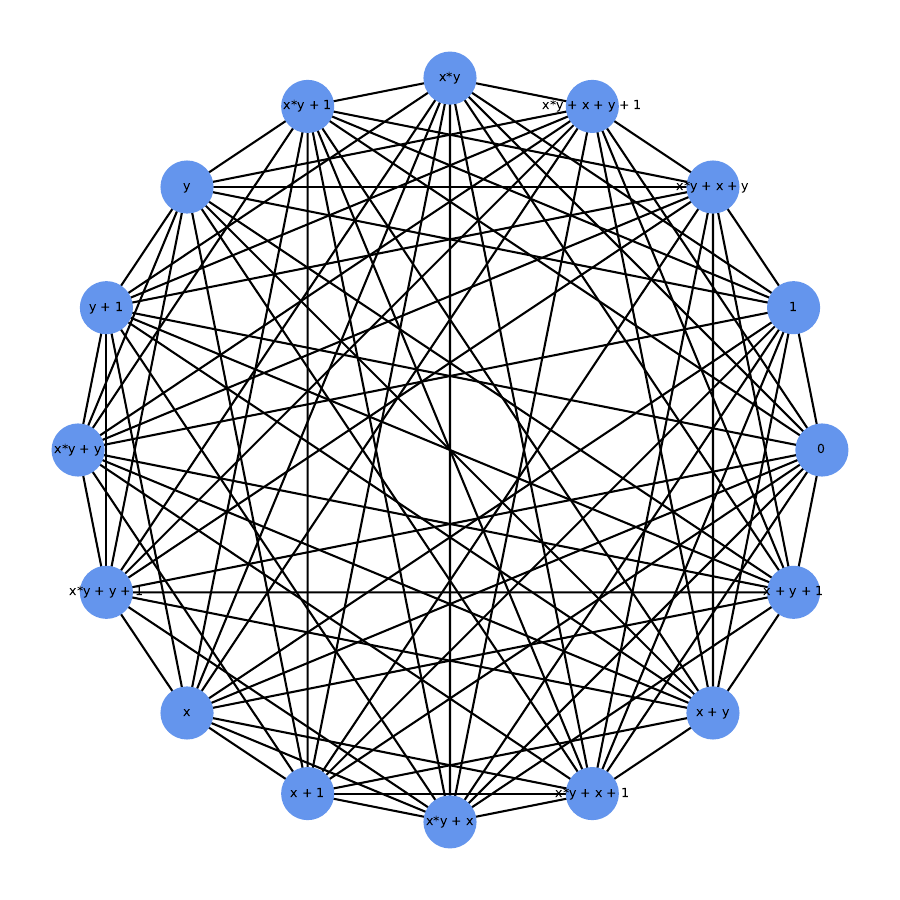} % Replace with your image file
        \caption{The gcd-graph $G_{R}(D)$ with $D=\{R, Rx, Rxy \}$}
        \label{fig:figure2}
    \end{subfigure}
    \caption{Two gcd-graphs over $\F_2[x,y]/(x^2,y^2)$}
    \label{fig:combined}
\end{figure}

\begin{expl} \label{expl:first}
Let $R=\F_2[x,y]/(x^2,y^2).$ Then $R$ is a Frobenius ring with $xy$ as the minimal element. Let $D=\{R, Rxy\}$. The characteristic polynomial of $G_{R}(D)$ is 
\[ (t - 9) (t + 7) (t - 1)^6  (t + 1)^8. \]
\cref{tab:dictionary1} records the value of $\lambda_r$ for $r \in R.$
\begin{table}[h]
    \centering
    \begin{tabular}{|c|l|}
        \hline
        \textbf{$\lambda_r$} & $r$ \\
        \hline
        9 & 0 \\
        \hline
        -1 & 1, $xy + x + 1$, $xy + x + y + 1$ \\
        & $x + y + 1$, $x + 1$, $xy + 1$, $xy + y + 1$, $y + 1$ \\
        \hline
        1 & $xy + x + y$, $xy + x$, $x + y$, $x$, $xy + y$, $y$ \\
        \hline
        -7 & $xy$ \\
        \hline
    \end{tabular}
    \caption{Spectrum of $G_{R}(D)$}
    \label{tab:dictionary1}
\end{table}
We can observe that $\lambda_{u}=-1$ for $u \in R^{\times}.$ On the other hand, $\lambda_u$ is of the form $4k+1$ if $u \in \mathfrak{m}.$ By \cref{prop:simplified_condition}, there exists PST between $0$ and $xy$ at time $\frac{\pi}{2}.$
\end{expl}

\begin{expl} \label{expl:second}
Let us consider another example. Let $R=\F_2[x,y]/(x^2, y^2)$ and $D=\{R, Rx, Rxy\}.$ In this case, the characteristic polynomial for $G_{R}(D)$ is 
\[ (t - 11) (t + 5)  (t - 3)^2 (t + 1)^{12} .\] 
The eigenvalues of $G_R(D)$ are described by \cref{tab:dictionary2}. 

\begin{table}[h]
    \centering
    \begin{tabular}{|c|l|}
        \hline
        \textbf{$\lambda_r$} & $r$ \\
        \hline
        11 & 0 \\
        \hline
        -1 & 1, $xy + x + y$, $xy + x + 1$ \\
        & $xy + x + y + 1$, $x + y$, $x + y + 1$, $x + 1$, $xy + y$ \\
        & $xy + 1$, $xy + y + 1$, $y$, $y + 1$ \\
        \hline
        3 & $xy + x$, $x$ \\
        \hline
        -5 & $xy$ \\
        \hline
    \end{tabular}
    \caption{Spectrum of $G_{R}(D)$}
    \label{tab:dictionary2}
\end{table}
As we can see, in this case $\lambda_{1} = \lambda_{xy+x+y}$ and hence the second condition in \cref{prop:simplified_condition} cannot happen. As a result, there is no PST on $G_{R}(D).$
\end{expl}

As we will see later, \cref{expl:second} represents a typical situation where the existence of  $r \in \mathfrak{m}$ such that $\lambda_r = \lambda_1$ is one of the main obstructions to the existence of PST on $G_{R}(D)$. It turns out that the socle series of $R$ will play a fundamental role in our study of this phenomenon. For this reason, we first recall its definition. 
\begin{definition} \label{def:soc}
    Let $M$ be an $R$-module. The socle modules $\text{soc}^{i}(M)$ are defined inductively as follows. 
    \begin{enumerate}
        \item $\text{soc}^{1}(M):=\text{soc}(M)$ is the sum of all minimal non-zero submodules of $M.$ By definition 
        \[ \soc^{1}(M) = \{a \in M \mid xa = 0 \quad \forall x \in \mathfrak{m} \} .\]
        \item For each $i \geq 2$, $\text{soc}^{i}(M)$ is the preimage of $\text{soc}^1(M/\text{soc}^{i-1}(M))$ in $M.$ 
    \end{enumerate}
\end{definition}

When $M=R$, we know that $\soc^1(R)=Re=\{0,e\}.$ For our investigation, $\soc^2(R)$ will be important. Since every ideal in $R$ contains $Re$, a minimal ideal $I$ containing $Re$ is necessarily principal; namely $I=Ra$ for some $a \in \mathfrak{m}$ such that $a \neq e.$ We remark that in this context, \textit{minimal} refers to a minimal ideal in the family of ideals that contain $Re$. In other words, $I$ is a pre-image of a minimal ideal in the quotient ring $R/Re$.

We provide below the necessary and sufficient condition for $Ra$ to be a minimal ideal containing $Re.$ 

\begin{lem} \label{lem:secondary_minimal}
\noindent
Let $a \neq e$ be an element of $R$. The following statements are equivalent. 
\begin{enumerate}
\item $Ra$ is a minimal ideal containing $Re.$
    \item  $xa \in Re$ for all $x \in \mathfrak{m}.$
    \item $a \in \Ann_{R}(\mathfrak{m^2}).$
    \item $a \in \soc^2(R).$
    \item $|R/\Ann_{R}(a)| = 4.$ 
\end{enumerate}
\end{lem}

\begin{proof}

%We first show that $(2) \iff (3) \iff (4).$ 
%The first equivalence follows from the fact that $\Ann_{R}(e)=\mathfrak{m}$ and $\Ann_{R}(\Ann_R(I))=I$ for every ideal $I$ in $R$. Since $a \neq e$, $(2) \implies (4)$ by the definition of $\soc^2(R).$ Finally, let us prove that $(4) \implies (2).$ In fact, since $a \in \soc^2(R)$, $a$ can be expressed as an $R$-linear combination of $a_i$ such that $Ra_i$ is a minimal ideal containing $Re.$ Since $x a_i \in Re$ for all $a_i$, the same holds for $a$ as well. 
Let $\bar{a}$ be  the image of $a$ in $R/\soc^1(R)$. We have that
\[
\begin{aligned}
a\in \soc^2(R) &\iff \bar{a} \in \soc^1(R/\soc^1(R))=\soc^1(R/Re) \\
&\iff x\bar{a}\in R/Re, \forall x\in \mathfrak{m}\\
&\iff xa \in Re, \forall x\in \mathfrak{m}.
\end{aligned}
\]
Hence $(2)\iff (4)$. 

Now we will show that $(2) \iff (3)$. Suppose $(3)$ holds and let $x,y\in \mathfrak{m}$. We have $ya\in Re$, hence $ya=be$, for some $b\in R$. Thus, $xya=bxe=b0=0$, since $x\in \mathfrak{m}=\Ann_R(e)$. Therefore $a\in \Ann_R(\mathfrak{m}^2)$. Conversely, suppose that $(4)$ holds. Let $x\in \mathfrak{m}$. Then $x^2\in \mathfrak{m}^2$ and $x^2a=0$. Hence $xa\in \Ann_R(\mathfrak{m})=\Ann_R(\Ann_R(Re))=Re.$

We will prove that $(1) \implies (2)  \implies (5) \implies (1).$ First, we can see that $(1) \implies (2)$ by the definition of the socle. Let us show that $(2) \implies (5).$ Since $\Ann_{R}(a) \subsetneq \mathfrak{m}$ and $|R/\mathfrak{m}|=2$, it is sufficient to show that $[\mathfrak{m}:\Ann_{R}(a)] =2.$ In fact, let $x_1, x_2 \in \mathfrak{m} \setminus \Ann_{R}(a).$ Then $x_1 a, x_2 a \in Re \setminus \{0\}=\{e\}$ by \cref{lem:local}. As a result, $(x_1-x_2)a =0$ and hence $x_1-x_2 \in \Ann_{R}(a).$ This shows that $[\mathfrak{m}:\Ann_{R}(a)] =2.$ Finally, let us show that $(5) \implies (1).$ Suppose to the contrary, that there exists an ideal $J$ such that $Re \subsetneq J \subsetneq Ra.$ Since $Re \subsetneq J$, we can find $b \in J$ such that $b \neq e.$ We then have $Re \subsetneq Rb \subsetneq Ra.$ By the duality between an ideal in $R$ and its annihilator in $R$, we have the reversed inclusion $\Ann_{R}(Ra) \subsetneq \Ann_{R}(Rb) \subsetneq \Ann_{R}(Re).$ This would imply that $|R/\Ann_{R}(a)| \geq 8$, which is a contradiction. 
\end{proof}
Next, we show that Ramanujan sums satisfy some modular congruences (see also \cite[Proposition 3.12]{minavc2023arithmetic} for some further congruences when $R$ is a quotient of the polynomial ring $\F_q[x]$). 
\begin{lem} \label{lem:congruence}
Let $r_1, r_2 \in \mathfrak{m}.$ Then for all $x \in R$, we have the following congruence 
\[ c(r_1, R/\Ann_{R}(x)) \equiv c(r_2, R/\Ann_{R}(x)) \pmod{4}. \] 
\end{lem}

\begin{proof}
We recall the definition of the Ramanujan sum  
\[ c(r, \Ann_{R}(x)) = \dfrac{\varphi(R/\Ann_{R}(x))}{\varphi(R/\Ann_{R}(rx))} \mu(R/\Ann_{R}(rx)). \] 
Furthermore, we know that $\mu(R/\Ann_{R}(rx)) \neq 0$ if and only $\Ann_{R}(rx) \in \{\mathfrak{m}, R \}.$ Consequently, $c(r, \Ann_{R}(x)) \in \{0, \varphi(R/\Ann_{R}(x)), - \varphi(R/\Ann_{R}(x)) \}.$ We consider two cases. 

\noindent \textbf{Case 1.} $Rx=Re$. In this case $x=e$ and $r_1x=r_2x=0.$ Therefore 
\[ c(r_1, R/\Ann_{R}(x)) = c(r_2, R/\Ann_{R}(x)). \]  

\noindent\textbf{Case 2.} $x \neq e.$ In this case $\Ann_{R}(x)$ is a proper ideal of $\mathfrak{m}$ and hence $|R/\Ann_{R}(x)| \geq 4.$ Consequently, $2\varphi(R/\Ann_{R}(x)) = |R/\Ann_{R}(x)| \equiv 0 \pmod{4}$. Therefore, if $r_1x, r_2x \in Re$, then both $c(r_1, \Ann_{R}(x)), c(r_2, \Ann_{R}(x)) \in \{ \varphi(R/\Ann_{R}(x)), - \varphi(R/\Ann_{R}(x)) \}$, and hence they are the same modulo $4.$ Similarly, if both $r_1x, r_2x$ are not in $Re$ then both these Ramanujan sums are zero. Finally, let us consider the case where exactly one of $r_1x, r_2x$ belongs to $Re.$ In this case, $x \not \in \soc^2(R)$. By \cref{lem:secondary_minimal}, we know that the order of $R/\Ann_{R}(x)$ is at least $8$ and therefore $\varphi(R/\Ann_{R}(x)) \equiv 0 \pmod{4}.$ As a result, $c(r_1, R/\Ann_{R}(x)) \equiv c(r_2, R/\Ann_{R}(x)) \equiv 0 \pmod{4}.$ 
\end{proof}

Let $\Isoc^2(R)$ be the set of ideals of the form $Ra$ where $a \in \soc^2(R).$ We remark that, by definition, $Re$ is an element of $\text{Isoc}^2(R).$ The following theorem explains the existence of PST in \cref{expl:first}. 

\begin{thm} \label{thm:local}
    Suppose that $|D \cap \Isoc^2(R)| \equiv 1 \pmod{2}.$ Then $G_{R}(D)$ has PST between $0$ and $e$ at time $\frac{\pi}{2}.$
\end{thm}

\begin{proof}
Let $D= \{Rx_1, \ldots, Rx_d \}$.  Let us first consider the case $Re \in D.$ In this case, we have 
\[ \lambda_0 = \sum_{i=1}^d \varphi(R/\Ann_{R}(x_i)). \]
By \cref{lem:secondary_minimal}, we know that $\varphi(R/\Ann_{R}(x_i))$ is a multiple of $4$ unless $x_i \in \soc^2(R).$ We conclude that 
\[ \lambda_0 = 1 +  \sum_{x_i \in \soc^2(R) \setminus \{e \}} \varphi(R/\Ann_{R}(x_i)) \equiv 1 \pmod{4}. \]
By \cref{lem:congruence}, we know that $\lambda_r \equiv 1 \pmod{4}$ for all $r \in \mathfrak{m}$. On the other hand, we have 
\[ \lambda_1 = \sum_{i=1}^d \mu(R/\Ann_{R}(x_i)) = -1. \]
This is because $\mu(R/\Ann_{R}(x_i))=0$ unless $x_i=e.$ By \cref{prop:simplified_condition}, there exists PST between $0$ and $e$ at time $\frac{\pi}{2}.$

Similarly, in the case $Re \not \in D$, we can show that $\lambda_r \equiv 2 \pmod{4}$ for all $r \in \mathfrak{m}$ and $\lambda_1 = 0.$ By \cref{prop:simplified_condition}, there exists PST between $0$ and $e$ at time $\frac{\pi}{2}.$
\end{proof}

We discuss a special case where the converse of \cref{thm:local} holds; namely when $R$ is a principal ideal local ring (also known as a finite chain ring). 
\begin{cor}
    Suppose that $(R,\mathfrak{m})$ is a principal ideal local ring with a generator $\alpha$ and residue field $\F_2$. Let $n$ be the smallest positive integer such that $\alpha^n=0.$ In this case, $\Isoc^2(R)= \{R\alpha^{n-2}, R \alpha^{n-1}\}.$  Furthermore, the converse of \cref{thm:local} is also true. In other words, a gcd-graph $G_R(D)$ has PST if and only if $|D \cap \{R\alpha^{n-1}, R\alpha^{n-2}\}| =1$. 
\end{cor}

\begin{proof}
    By \cref{thm:local} if $|D \cap \{R\alpha^{n-1}, R\alpha^{n-2}\}| =1$, then there exists PST between $0$ and $e=\alpha^{n-1}$ at time $\frac{\pi}{2}.$ Suppose that $|D \cap \{R\alpha^{n-1}, R\alpha^{n-2}\}| \in \{0, 2\}$. We claim that there is no PST on $G_{R}(D).$ We consider two cases. 

    \noindent\textbf{Case 1.} $R \alpha^{n-1}, R \alpha^{n-2} \in D.$ In this case, we have $\lambda_1 = -1$ and
    \begin{align*}
     \lambda_{\alpha} &= \frac{\varphi(R/\Ann_{R}(\alpha^{n-2}))}{\varphi(R/\Ann_{R}(\alpha^{n-1}))} \mu(R/\Ann_{R}(\alpha^{n-1}))+ \frac{\varphi(R/\Ann_{R}(\alpha^{n-1}))}{\varphi(R/\Ann_{R}(\alpha^{n}))} \mu(R/\Ann_{R}(\alpha^{n})) 
     \\ &= (-2)+1 = -1. 
    \end{align*}
    By \cref{prop:simplified_condition}, there is no PST on $G_R(D).$
    
    \noindent \textbf{Case 2.} $R \alpha^{n-1}, R \alpha^{n-2} \not \in D.$ In this case, $\lambda_1 = \lambda_{\alpha}=0.$ As a result, \cref{prop:simplified_condition} also implies that there is no PST on $G_R(D).$
\end{proof}

We discuss another case in which the converse of \cref{thm:local} also holds. 

\begin{cor}
    Suppose that $D = \{R, Ra \}$ for some $a \in \mathfrak{m}.$ Then the converse of \cref{thm:local} holds. In other words, $G_{R}(D)$ has PST between $0$ and $e$ if and only if $|D \cap \Isoc^2(R)| =1$. 
\end{cor}

\begin{proof}
    The forward direction is a consequence of \cref{thm:local}. Let us prove the other direction, that is, if $|D \cap \Isoc^2(R)|=0$ then $G_{R}(D)$ has no PST. In fact, in this case we have $\lambda_1 =0.$ Furthermore, since $a \not \in \soc^2(R)$, there exists $b \in \mathfrak{m}$ such that $ab \not \in Re.$ We then see that
    \[ \lambda_b = \frac{\varphi(R)}{\varphi(R/\Ann_{R}(b))} \mu(R/\Ann_{R}(b)) + \frac{\varphi(R/\Ann_{R}(a)}{\varphi(R/\Ann_{R}(ab))} \mu(R/\Ann_{R}(ab)) =0.\]
    The last equality follows from the fact that both $\Ann_{R}(b)$ and $\Ann_{R}(ab)$ are proper ideals of $\mathfrak{m}.$ By \cref{prop:simplified_condition}, we conclude that there is no PST on $G_{R}(D).$
\end{proof}

When $R$ is an $\F_2$-algebra, then, as an abelian group, $R$ is isomorphic to $\F_2^r$ for some $r \geq 0.$ In this case, $G_{R}(D)$ is an example of a cube-like graph. The study of PST on these graphs has extensive literature (see \cite{cheung2011perfect} and the references therein for further discussion). We will show below that \cref{thm:local} is consistent with \cite[Theorem 2.3]{cheung2011perfect}. To do so, we first need the following lemma.

\begin{lem}  \label{lem:sum_elements}
Let $a \in R.$ Then 
  \[
\sum_{b \sim a } b =
\begin{cases}
e & \text{if } R/\Ann_{R}(a) \in \{\F_2, \F_2[x]/x^2\}, \\
0 & \text{otherwise. } 
\end{cases}
\]
\end{lem} 
\begin{proof}
    We first prove this statement for when $a \in R^{\times}.$ If $2 \neq 0 $ in $R$ then elements of $R^{\times}$ can be partitioned into pairs $\{u, -u\}$. As a result $\sum_{b \in R^{\times}} b =0.$ Let us consider the case the characteristic of $R$ is $2.$ If $R=\F_2$ then  $e=1$ and hence the statement is true. Let us consider the case $R \neq \F_2$. In this case, the elements of $R^{\times}$ can be partitioned into pairs $\{u, u+e\}.$ For each pair, the sum is $e$. Therefore 
    \[ \sum_{b \in R^{\times}} b = \frac{1}{2} \varphi(R) e. \]
    This is $0$ unless $\varphi(R)=2.$ Since the characteristic of $R$ is $2$, $\varphi(R)=2$ can only happen when $R=\F_2$ or $\F_2[x]/x^2.$

    Let us now consider the general case. First,  since $a \neq 0$, we can find $e_1$ such that $e = a e_1.$ While $e_1$ is not unique, it is unique in $R/\Ann_{R}(a).$ We claim that $Re_1$ is the minimal ideal in $R/\Ann_{R}(a)$. (This gives another proof for the fact that $R/\Ann_{R}(a)$ is a Frobenius ring. A different proof using a non-degenerate linear functional was previously discussed in \cite[Lemma 4.3]{nguyen2025gcd}.) In fact, let $u \in R/\Ann_{R}(a)$ such that $u$ is neither a unit nor $0$. Let $\hat{u}$ be a lift of $u$ to $R.$ Since $u \neq 0$, $\hat{u}a \neq 0$ and hence we can write $e= v(ua).$ We can then see that $e_1 = \hat{u}v $ in $R/\Ann_{R}(a).$ This shows that $Re_1$ is the unique minimal ideal in $R/\Ann_{R}(x).$

    By \cite[Corollary 2.7]{nguyen2025gcd} we have 
    \[ \sum_{b \sim a} b = a \sum_{u \in (R/\Ann_{R}(a))^{\times}} u. \]
    By the calculation done in the previous part, we know that modulo $\Ann_{R}(a)$

  \[
\sum_{u \in (R/\Ann_{R}(a))^{\times}} u=
\begin{cases}
e_1 =e& \text{if } R/\Ann_{R}(a) \in \{\F_2, \F_2[x]/x^2\}. \\
0 & \text{otherwise. } 
\end{cases}
\]

The statement follows  from this.     
\end{proof}
We can now relate \cref{thm:local} and \cite[Theorem 2.3]{cheung2011perfect}.

\begin{prop}
    Suppose that $R$ is an $\F_2$-algebra. Let $S$ be the generating set of $G_R(D)$ associated with $D$. Then the following statements are equivalent.
    \begin{enumerate}
        \item $|D \cap {\Isoc}^2(R)| \equiv 1 \pmod{2}.$
        \item $\sum_{s \in S} s \neq 0.$
    \end{enumerate}
    Furthermore, if one of these conditions hold, there exists PST between $0$ and $e$ at time $t=\frac{\pi}{2}.$
\end{prop}

\begin{proof}
By the definition of $S$,
\[ \sum_{s \in S} s = \sum_{i=1}^d \left[\sum_{b \sim x_i} b  \right]. \]
By \cref{lem:sum_elements}, the sum 
\[ \sum_{b \sim x_i} b = 
\begin{cases}
e & \text{if } R/\Ann_{R}(x_i) \in \{\F_2, \F_2[x]/x^2\}, \\
0 & \text{otherwise. } 
\end{cases} 
\]     

Furthermore, \cref{lem:secondary_minimal} implies that $R/\Ann_{R}(x_i) \in \{\F_2, \F_2[x]/x^2\}$ if and only if $x_i \in \soc^2(R).$ We conclude that 
\[ \sum_{s \in S} s = |D \cap \Isoc^2(R)|e. \] 
This is nonzero if and only if $|D \cap \Isoc^2(R)|$ is odd. 
\end{proof}

In all the examples that we consider, the converse of \cref{thm:local} always holds. More specifically, in the case where $|D \cap \text{Isoc}^2(R)| \equiv 0 \pmod{2}$, we can always find $a \in \mathfrak{m}$ such that $\lambda_a = \lambda_1$ (note that $\lambda_1$ is $-1$ or $0$ depending on whether $Re \in D$ or not). It would be interesting to prove or disprove this statement.

\begin{rem}

We note that since the complete graph $K_n$ has no PST unless $n=2$, the converse of \cref{thm:local} would imply that if $R \neq \F_2$, $|\Isoc^2(R)|$ must be an even number (otherwise, if we take $D$ to be the union of all principal ideals in $R$, the converse of \cref{thm:local} would imply that the complete graph $K_{|R|}$ has PST). Fortunately, we can prove this statement unconditionally. Let $Re_1, Re_2, \ldots, Re_k$ be the elements of $\Isoc^2(R)\setminus \{Re\}.$ Then for $Re_i$ we have exactly $4$ elements. In addition, for $i \neq j$, $|Re_i \cap Re_j| = |Re|=2$. We conclude that 
\[ |\soc^2(R)| = |Re|  + 2k = 2+2k. \]
When $R \neq \F_2$, $\soc^{1}(R)$ is a proper subgroup of $\soc^2(R).$ As a result, $4 \mid \soc^2(R).$ This implies that $k$ must be odd. In other words, $|\Isoc^2(R)|$ must be even.  
\end{rem}
\iffalse

%
\begin{lem}
    Let $V$ be a vector space over $\F_2.$ Let $S$ be a subset of $V$ such that $|S|$ is even. Then, there exists a functional $f \in V^{*}= \Hom(V, \F_2)$ such that $|\ker(f) \cap S| = \frac{|S|}{2}.$
\end{lem}

\begin{question}
Let $S$ be a subset of $\F_2^{r}$ such that $|S| \leq 2^{r-1}.$ Is it true that the set $A=|S-S| = \{s_1-s_2| s_1, s_2 \in S \}$ cannot be equal to $\F_2^r$? 
\end{question}
\fi 
\section{\textbf{PST on unitary Cayley graphs}} \label{sec:unitary}
We recall that for a finite ring $R$, the unitary Cayley graph $G_{R}(1)$ associated with $R$ is precisely the gcd-graph $G_{R}(\{R\}).$ In other words, it is the graph whose vertex set is $R$ and two vertices $a,b$ are adjacent if $a-b \in R^{\times}.$ In this section, we classify all finite Frobenius rings whose associated unitary Cayley graph has PST. We remark that various authors study a similar problem in the context of discrete-time quantum walks (for example, see \cite{bhakta2025state}). Thus, our results provide new insights that are independent of their findings. Additionally, after completing this article, we learned via MathSciNet that \cref{thm:unitary} has been obtained in \cite[Theorem 2.5]{pst-gcd-new} using a different approach, namely via matrix analysis. Our approach discussed here is more number-theoretic in nature, and we hope that it will provide a new toolbox for researchers in the field. For this reason, we have decided to retain this section.

We recall that the spectrum of $G_{R}(1)$ is the multiset
\[ \left\{ \lambda_g = c(g, R) \right \} = \left\{ \frac{\varphi(R)}{\varphi(R/\Ann_{R}(g))} \mu(R/\Ann_{R}(g)) \right \}.\]
Let $R=(\prod_{i=1}^d S_i) \times R_2$ be the Artin-Wedderburn decomposition of $R$ as described in \cref{prop:upper_bound}. Since $\mu$ and $\varphi$ are multiplicative, if we write $g=(g_1, g_2, \ldots, g_d, h)$ then 
\[ \lambda_{g} = c(h, R_2) \prod_{i=1}^d c(g_i, S_i). \]

Recall that $\Delta$  is the abelian subgroup of $(R,+)$ generated by $r_1-r_2$ where $\lambda_{r_1}=\lambda_{r_2}$. 
\begin{prop}
\noindent
\begin{enumerate}
    \item  Suppose that there are at least two local factors $T$ of $R$ such that $\mu(T)=0.$ Then $\Delta = R.$
    \item Let $R=(\prod_{i=1}^d S_i) \times R_2$ as described above. If $\mu(R_2) = 0$ then $\Delta= R.$
\end{enumerate}
\end{prop}

\begin{proof}
Let us prove the first statement. Let us write $R=T_1 \times T_2 \times T_3$ where $T_1, T_2$ are local rings such that $\mu(T_1)=\mu(T_2)=0.$ Let $t=(t_1, t_2, t_3)$ be an arbitrary element in $R.$ We have 
\[ c(1, t_2+1, t_3) = c(1, T_1) c(t_2+1, T_2) c(t_3, T_3)=0.\]
Similarly $c(1-t_1, 1, 0)=0.$ This shows that $(t_1, t_2, t_3)=(1, t_2+1, t_3)-(1-t_1, 1, 0) \in \Delta.$ This shows that $\Delta = R.$

Let us now prove the second statement. For each $a \in \prod_{i=1}^d S_i$ we have $c((a, 1),R)=c((0,1),R)=0.$ Consequently $(a,0) \in \Delta$ since $(a,0) =(a,1)-(0,1) \in \Delta$. Since $J(\prod_{i=1}^d S_i) \times R_2 \subset \Delta$, we know that $(a,b) \in \Delta$ for all $a \in \prod_{i=1}^d S_i$ and $b \in R_2.$ In other words, $\Delta = R.$
    \end{proof}
By \cref{cor:orthogonal}, we have the following. 
\begin{cor}
Suppose that there exists PST on $G_{R}(1).$ Then the following conditions hold. 
\begin{enumerate}
    \item $R$ has at most one local factor $T$ such that $\mu(T)=0.$
    \item $\mu(R_2) \neq 0.$
\end{enumerate}
In other words, $R$ must have the form $S_1 \times \prod_{i=1}^{d-1} \F_2 \times \prod_{i=1}^r \F_{q_i} $ where $S_1$ is a Frobenius local ring with residue field $\F_2$ and $\F_{q_i}$ are finite fields with $q_i \geq 3.$
\end{cor}

We can impose further restrictions on $G_{R}(1)$ when there exists PST on it. In fact, $G_{R}(1)$ has exactly $2^{d-1}$ connected components, each of which is isomorphic to $G_{S_1 \times \prod_{i=1}^r \F_{q_i} }(1).$ Since PST cannot exist between vertices in different connected components, we can safely assume that $d=0$ and $R=S_1 \times \prod_{i=1}^{r} \F_{q_i}$ where $(S_1, \mathfrak{m}_1)$ is a local ring with residue field $\F_2.$

\begin{thm} \label{thm:unitary}
    Suppose that there exists PST on $G_{R}(1)$ with $R=S_1 \times \prod_{i=1}^r \F_{q_i}$. Then the following conditions must hold. 
\begin{enumerate}
    \item $S_1 \in \{\F_2, \Z/4, \F_2[x]/x^2 \}.$
    \item $q_i \equiv 0 \pmod{4}$ for all $i.$
\end{enumerate}
Conversely, suppose that these conditions hold. Then there exists PST between $(0,0)$ and $(e_1, 0)$ at time $\frac{\pi}{2}.$ Here $e_1$ is the minimal element of $S_1.$
\end{thm}

\begin{proof}
    Let us prove the necessary conditions. Suppose that there exists PST between $0$ and some $s \in R.$ By \cref{prop:upper_bound}, $s$ must equal to $(e_1, 0, \ldots, 0).$ In this case, the conditions in \cref{prop:condition} can be described explicitly as 
\begin{equation} \label{eq:conditions}
    \begin{cases}
    (\lambda_{(x_1, a_1, \ldots, a_r)} - \lambda_{(x_2, b_1, \ldots, b_r)}) \frac{t}{2 \pi} + \frac{1}{2} \in \Z  \text{if } (x_1-x_2) \not \in \mathfrak{m}_1 \\ 
    (\lambda_{(x_1, a_1, \ldots, a_r)} - \lambda_{(x_2, b_1, \ldots, b_r)}) \frac{t}{2 \pi}  \in \Z  \text{if  } (x_1-x_2)  \in \mathfrak{m}_1.
\end{cases} 
\end{equation}
This follows from the fact that $\frac{\psi((e_1,0, \ldots, 0))}{n}= \frac{1}{2}$ as shown in \cref{lem:local}. We also note that 
\begin{align*} 
\lambda_{(x_1, a_1, \ldots, a_r)} &= \frac{\varphi(S_1)}{\varphi(S_1/\Ann_{S_1}(x_1))} \mu(S_1/\Ann_{S_1}(x_1)) \prod_{i=1}^r \frac{\varphi(\F_{q_i})}{\varphi(\F_{q_i}/\Ann_{\F_{q_i}}(a_i))} \mu(\F_{q_i}/\Ann_{\F_{q_i}}(a_i)) \\ 
&= \frac{\varphi(S_1)}{\varphi(S_1/\Ann_{S_1}(x_1))} \mu(S_1/\Ann_{S_1}(x_1)) \prod_{a_i =0} (q_i-1) \prod_{a_i \neq 0 }(-1) \\
&= \frac{\varphi(S_1)}{\varphi(S_1/\Ann_{S_1}(x_1))} \mu(S_1/\Ann_{S_1}(x_1)) (-1)^r \prod_{a_i =0} (1-q_i).
\end{align*}

First, we claim that $S_1 \in \{\F_2, \Z/4, \F_2[x]/x^2 \}.$ If $S_1=\F_2$ then we are done. Therefore, to avoid tautology, we will assume that $S_1 \neq \F_2.$ We first claim that $\mathfrak{m}_1=S_1e_1$. In fact, suppose that it is not the case. Then we can find $x \in \mathfrak{m}_1 \setminus Re$ (since we are assuming that $S_1 \neq \F_2$, $S_1e_1 \subset \mathfrak{m}_1).$ We then have $\Ann_{S_1}(x) \neq \mathfrak{m}_1$ and hence $\lambda_{(x,0,\ldots,0)}=0.$ Similarly, we also have $\lambda_{(1,0,\ldots, 0)}=0$. These two equalities contradict the first condition. We conclude that $\mathfrak{m}_1 = S_1 e_1.$ This would imply that $\mathfrak{m}_1^2=0$ and hence $1 \in \Ann_{S_1}(\mathfrak{m}_1^2).$ By \cref{lem:secondary_minimal}, we conclude that $|S_1|=4$ and hence $S_1 \in \{ \Z/4$, $\F_2[x]/x^2 \}.$ 

Next, we claim that $q_i \equiv 0 \pmod{4}$ for each $1 \leq i \leq r.$

\textbf{Case 1.} Let's first consider the case $S_1=\F_2.$ We have $\lambda_{(1, 1, 1,\ldots, 1)}=(-1)^{r+1}$,  $\lambda_{(0,1,\ldots, 1)} =(-1)^r$, and $\lambda_{(1, 0, 1, \ldots, 1)}=(-1)^{r+1}(1-q_1).$ The first condition then implies that $2(\frac{t}{2 \pi})+\frac{1}{2} \in \Z.$ The second condition then implies that $q_1 \frac{t}{2 \pi} \in \Z.$ This condition necessarily implies that $q_1$ must be even. By the same argument, $q_i$ is even for all $1 \leq i \leq r.$

\textbf{Case 2.} Let us consider the case $S_1 \in \{\Z/4, \F_2[x]/x^2 \}.$ In this case $\lambda_{(e, 1, 1,\ldots, 1)}=2(-1)^{r+1}$, $\lambda_{(1,1,\ldots, 1)}=0$,  and $\lambda_{(e,0,1\ldots,1)}= 2(-1)^r (1-q_1).$ The first condition would imply that $2(\frac{t}{2 \pi})+\frac{1}{2} \in \Z.$ The second condition implies that $2(-1)^{r+1}(1-q_1) \frac{t}{2 \pi} \in \Z. $ We conclude that $q_1$ must be even. By the same argument, $q_i$ are even for all $1 \leq i \leq r.$

Conversely, suppose that all the above conditions are satisfied. We claim that there exists PST between $0$ and $(e_1, 0, \ldots, 0)$ at time $\frac{\pi}{2}.$ We consider two separate cases as above. 

In the first case, namely when $S_1 = \F_2$, we can see that $\lambda_{(1,a_1, a_2, \ldots, a_r)} \equiv (-1)^{r+1} \pmod{4}$ and $\lambda_{(0, a_1, \ldots, a_r)} \equiv (-1)^r \pmod{4}$ for every $(a_1, a_2, \ldots, a_r) \in \prod_{i=1}^r \F_{q_i}.$ We can check that the conditions described in \cref{eq:conditions} are satisfied when $t = \frac{\pi}{2}.$

In the second case where $S_1 \in \{\Z/4, \F_2[x]/x^2\}$ we can see that 
\[ \lambda_{(0,a_1, \ldots, a_r)} \equiv \lambda_{(e_1, a_1, \ldots, a_r)} \equiv 2(-1)^r \pmod{4} .\] 
On the other hand, if $u \in S_1^{\times}$ then $\lambda_{(u_1, a_1, \ldots, a_r)} \equiv 0 \pmod{4}$. These congruences show that all conditions described by \cref{eq:conditions} are met when $t = \frac{\pi}{2}.$
\end{proof}

When $R$ is a finite quotient of $\F_q[x]$, it is known that $R$ is a Frobenius ring (see \cite[Corollary 6.8]{minavc2024gcd}). In this case, \cref{thm:unitary} implies the following  corollary. 

\begin{cor}
    Let $f \in \F_q[x]$ be a monic polynomial. Let $G_{f}(1):=G_{\F_q[x]/f}(1)$ be the unitary Cayley graph associated with $\F_q[x]/f.$ Then there exists PST on $G_{f}(1)$ if and only if the following conditions hold 
    \begin{enumerate}
        \item $q=2$.
        \item Let $f=x^a (x+1)^b g(x)$. Then $g$ is squarefree, $\gcd(g, x(x+1))=1$, $2 \geq \max\{a,b\} \geq 1$, and $(a,b) \neq (2,2).$
    \end{enumerate}
\end{cor}

\section*{\textbf{Acknowledgements}}
The first-named author would like to thank Professor Chris Godsil and Professor Marko Petkovi\'c for some helpful correspondence. He is also thankful to Dr. Ricardo Buring for his help with Sagemath.  The second-named author gratefully acknowledges the Vietnam Institute for Advanced Study in Mathematics (VIASM) for hospitality and support during a visit in 2025. Finally, we thank Professor Yotsanan Meemark for sending us a copy of \cite{pst-gcd-new}.

\providecommand{\bysame}{\leavevmode\hbox to3em{\hrulefill}\thinspace}
\providecommand{\MR}{\relax\ifhmode\unskip\space\fi MR }
% \MRhref is called by the amsart/book/proc definition of \MR.
\providecommand{\MRhref}[2]{%
  \href{http://www.ams.org/mathscinet-getitem?mr=#1}{#2}
}

\end{document}